\documentclass{amsart}
\usepackage[utf8]{inputenc}
\usepackage{amsmath, amsthm, amssymb}
\usepackage{tikz}
\usetikzlibrary{matrix}

\newtheorem{theorem}{Theorem}
\newtheorem{lemma}{Lemma}
\newtheorem{corollary}[lemma]{Corollary}
\newtheorem{proposition}[lemma]{Proposition}
\newtheorem{definition}[lemma]{Definition}\theoremstyle{definition}
\newtheorem{remark}[lemma]{Remark}\theoremstyle{remark}
\newtheorem{problem}[lemma]{Problem}

\begin{document}

\title{Unipotent Factorization of Vector Bundle Automorphisms}

\author{Jakob Hultgren}
\address{Jakob Hultgren \\
University of Maryland\\
Department of Mathematics\\
4176 Campus Drive - William E. Kirwan Hall\\
College Park, MD 20742-4015\\
US
}
\email{hultgren@umd.edu}

\author{Erlend F. Wold} 
\address{Erlend F. Wold \\
Matematisk Institutt \\ Universitetet i Oslo \\ Postboks 1053 Blindern \\ 0316 Oslo \\ Norway
}
\email{erlendfw@math.uio.no}

%\subjclass[2010]{15A54; 15A23; 55N15; 32A38; 19B99}
\subjclass[2000]{15A54; 15A23; 55N15; 32A38; 19B99}

\maketitle

%\date{}

\begin{abstract}
We provide unipotent factorizations of vector bundle automorphisms of real and complex vector bundles over finite dimensional locally finite CW-complexes. This generalises work of Thurston-Vaserstein and Vaserstein for trivial vector bundles. We also address two symplectic cases and propose a complex geometric analog of the problem in the setting of holomorphic vector bundles over Stein manifolds.

\end{abstract}

\section{Introduction}
By elementary linear algebra any matrix in $\mathrm{SL}_k(\mathbb R)$ or $\mathrm{SL}_k(\mathbb C)$ can be written as a product of elementary matrices $\mathrm{id}+\alpha e_{ij}$, i.e., matrices with ones on the diagonal and at most one non-zero element outside the diagonal. Replacing $\mathrm{SL}_k(\mathbb R)$ or $\mathrm{SL}_k(\mathbb C)$ by $\mathrm{SL}_k(R)$ where $R$ is the ring of continuous real or complex valued functions on a topological space $X$, we arrive at a much more subtle problem. This problem was adressed Thurston and Vaserstein \cite{TV86} in the case where $X$ is the Euclidean space and more generally by Vaserstein \cite{Vas88} for a finite dimensional normal topological space $X$. In particular, Vaserstein \cite{Vas88}  proves that for any finite dimensional normal topological space $X$, and any continuous map $F:X\rightarrow \mathrm{SL}_k(\mathbb R)$ for $k\geq 3$ (and $\mathrm{SL}_k(\mathbb C)$ for $k\geq 2$, respectively) which is homotopic to the constant map $x\mapsto \mathrm{id}$, there are continuous maps $E_1,\ldots, E_N$ from $X$ to the space of elementary real (resp. complex) matrices, such that 
$$ F = E_N\circ \cdots \circ E_1. $$ 

More recently, the problem has been considered by Doubtsov and Kutzschebauch \cite{DK19}. \

Recall that a map $S$ on a vector space is unipotent if $(S-\mathrm{id})^m=0$ for some $m$. Note that, again by elementary linear algebra, providing a factorisation in terms of elementary matrices is equivalent to providing a factorisation in terms of upper and lower triangular unipotent matrices, i.e., matrices of the form 
$$\mathrm{id} + \sum_{i<j} \alpha_{ij}e_{ij} \textnormal{ and } \mathrm{id} + \sum_{i>j} \alpha_{ij}e_{ij},$$
respectively. 

Moreover, note that an element in $\mathrm{SL}_k(R)$ where $R$ is the ring of continuous real valued functions on a topological space $X$ is equivalent to a continuous $\mathrm{SL}_k(\mathbb R)$-valued function on $X$. Similarly, an element in $\mathrm{SL}_k(R)$ where $R$ is the ring of continuous complex valued functions on a topological space $X$ is equivalent to a continuous $\mathrm{SL}_k(\mathbb C)$-valued function on $X$.

In this paper we will consider the following problem: Let $X$ be a CW-complex and 
$$\pi:V\rightarrow X$$
a real or complex vector bundle over $X$ of rank $k$. Moreover, let $S$ be a vector bundle automorphism of $V$ with constant determinant equal to one, in other words $S$ is a homeomorphism $V\rightarrow V$ such that $\pi \circ S = \pi$ and $S$ restricts to a linear map of determinant one on each fiber $\pi^{-1}(x), x\in X$. We will call such an object a \emph{special} vector bundle automorphism. Given this data, can $S$ be factored into a composition of unipotent vector bundle automorphisms, i.e., vector bundle automorphisms that restrict to unipotent maps on the fibers? 

Fixing a local trivialization $\rho:V|_U\rightarrow U\times \mathbb R^k$ (resp. $\rho:V|_U\rightarrow U\times \mathbb C^k$) of $V$ over a subset $U\subset X$ we get that a special vector bundle automorphism $S$ satisfies
$$ \rho \circ S|_U \circ \rho^{-1}(x,y) = (x,F(x)) $$
where $F$ is a $SL_k$-valued function on $U$. In other words, special vector bundle automorphisms can be represented locally as $\mathrm{SL}_k$-valued functions. In this sense the problem considered by Thurston and Vaserstein can be seen as a special case (given by the cases when $V=X\times \mathbb R^k$ and $V=X\times \mathbb C^k$) of the problem considered in this paper.

We will say that a special vector bundle automorphism $S$ is \emph{nullhomotopic} if there is a homotopy $(t,x)\in [0,1]\times X \mapsto S_t(x)$ such that $S_1=S$, $S_0$ is the identity map on $V$ and $S_t$ is a special vector bundle automorphism for each $t\in [0,1]$. Our first result is the following.
\begin{theorem}\label{thm:1}
Let $X$ be a locally finite finite dimensional CW-complex and let $\pi:V\rightarrow X$ be a real (resp. complex) vector bundle 
of rank $k$. 
Assume that $k\geq 3$ (resp. $k\geq 2$) and let $S$ be a nullhomotopic special vector bundle automorphism of $V$. Then there exist unipotent vector bundle automorphisms
$E_1,\ldots, E_N$ such that 
$$
S=E_N\circ\cdot\cdot\cdot\circ E_1.
$$
Moreover, for all $j$ we have that 
$(E_j-\mathrm{id})^2=0$ and $\mathrm{rank}(E_j-\mathrm{id})\in\{0,1\}$.
\end{theorem}

\begin{remark}
Note that for a vector space
$V$ and for $E\in\mathrm{Iso}(V)$,
the conditions that $(E-\mathrm{id})^2=0$
and $\mathrm{rank}(E-\mathrm{id})\in\{0,1\}$ are equivalent to 
the existence of $\alpha\in V^*$
\begin{equation}\label{shear}
E(u)=u +\alpha(u)\cdot v
\end{equation}
with $\alpha(v)=0$. We will call isomorphisms on the form \eqref{shear}
\emph{elementary}. 
\end{remark}

\begin{remark}
We make a short remark about the assumption in Theorem~\ref{thm:1} that $k\geq 3$ in the case of a real vector bundle. The fact that sets the case $k=2$ apart is that the fundamental group of $\mathrm{SO}_2(\mathbb R)$ is infinite. In \cite{Vas88}, Vaserstein provides an example of a $\mathrm{SL}_2(\mathbb R)$-valued function on $\mathbb R$ which does not factor into a product of maps into the space of real elementary matrices.
\end{remark}

\begin{remark}
It would be interesting to address the question of how many factors are needed in Theorem~\ref{thm:1}, in a similar spirit as the results in \cite{KS19}
\end{remark}

If $V=\mathbb R^k$ (resp. $\mathbb C^k$) it is customary to require an elementary 
isomorphism to satisfy the additional requirements that $v=e_j$ for some 
$j$, and $\alpha(v)=\lambda\cdot \langle v,e_i\rangle$ for some $i\neq j, \lambda\in \mathbb R$ (resp. $\lambda\in\mathbb C$), 
where $\{e_1,...,e_k\}$ denotes the set of standard basis vectors in $\mathbb R^k$. Assuming 
the necessary additional structure on $V$, we may improve Theorem \ref{thm:1} to obtain 
an analogous stronger form of decomposition by elementary matrices. 

\begin{theorem}\label{thm:2}
Suppose in addition to the data in Theorem \ref{thm:1} that $V$
splits globally into line bundles $\mathcal L=\{L_1,...,L_k\}$. Then, in addition to 
the conclusions in Theorem  \ref{thm:1} we may achieve that 
for any point $x\in X$ and for any frame $\{e_1(x),\ldots,e_k(x)\}$ for $V_x$ such that 
$e_s\in L_s$ for all $s\in \{1,\ldots,k\}$,
that $E_j(x)$ is of the form 
\begin{equation}\label{L-el}
E_j(x)(u)=u + \alpha_j(u)\cdot e_{s_j}(x),
\end{equation}
where $\alpha_j(e_s)\neq 0$ if and only if $s=s_0$ for some fixed $s_0\neq s_j$.
\end{theorem}
\begin{remark}
We will call a vector bundle automorphism satisfying \eqref{L-el} at any point $\mathcal L$-elementary. 
\end{remark}
\begin{remark}
For any vector bundle $V\rightarrow X$ we may consider the flag bundle $Y=\mathrm{Fl}(V)$
associated to $V$, with a map $p:Y\rightarrow X$ such that $p^*V$ splits into a family $L$ of line bundles (see e.g. Hatcher, \cite{H2003}).
Then, by Theorem~\ref{thm:2}, $p^*S$ factors into a composition of finitely many $\mathcal L$-elementary vector bundle automorphisms. 
\end{remark}

\subsection{The curvature tensor of a Riemannian manifold}
We will now present an interesting special case of Theorem~\ref{thm:1}. Assume that $X$ is a smooth manifold equipped with a Riemannian metric, let $\pi:TX\rightarrow X$ be the tangent bundle of $X$, $\nabla$ its Levi-Civita connection, and $R$ its Riemann curvature tensor. A vector bundle endomorphism of $TX$ is a smooth map $V\rightarrow V$ such that $\pi \circ S = \pi$ and $S$ restricts to a linear map on each fiber $\pi^{-1}(x), x\in X$. Fixing two smooth vector fields $U,V$ on $X$, we get the Riemann curvature endomorphism of $TX$ given by
\begin{equation} \label{eq:CurvatureMap} R(U,V) = \nabla_U\nabla_V - \nabla_V\nabla_U - \nabla_{[U,V]}. \end{equation}
The restriction of this map to any fiber in $TX$ defines a skew-symmetric map with respect to the Riemannian metric. It follows that applying the exponential map to this, we get a vector bundle automorphism of $TX$ which preserves the Riemannian metric, i.e., it can be locally represented by an orthogonal matrix. In particular, it has determinant one. Applying Theorem~\ref{thm:1} to this gives the following:
\begin{corollary}\label{cor:Riemannian}
Let $X$ be a Riemannian manifold, $R(\cdot,\cdot)$ its Riemann curvature tensor, $\nabla$ its Levi-Civita connection, and let $U,V$ be two smooth vector fields on $X$. Then the vector bundle automorphism  of $TX$ given by 
$$ \exp(R(U,V)) = \exp(\nabla_U\nabla_V - \nabla_V\nabla_U - \nabla_{[U,V]})$$
admits a factorisation in terms of unipotent vector bundle automorphisms. 
\end{corollary}

\subsection{Real and complex symplectic vetor bundles}
In \cite{IKL19}, Kutzschebauch, Ivarsson, and L\o w studies factorization properties of symplectic matrices over various rings of complex valued functions. We will now present two results in this direction, one for complex vector bundles and one for real vector bundles. Recall that a symplectic form on a complex or real vector bundle $V$ is a non-degenerate section $\omega$ of $\Lambda ^2V^*$, i.e., 
for each point $x\in X$ and any $u\in V_x$ we have that $\omega_x(u,\cdot)$ is 
not identically zero. We will say that a vector bundle automorphism of $V$ is symplectic if it preserves $\omega$. 

We will start with the complex case. Recall that the compact symplectic group is the group of complex $2k\times 2k$-matrices given by 
$$ \mathrm{USp}_{2k} = \mathrm{Sp}_{2k}(\mathbb C)\cap \mathrm{U}_{2k} $$
where $\mathrm{Sp}_{2k}(\mathbb C)$ is the group of symplectic complex $2k\times 2k$-matrices, i.e., maps from $\mathbb C^{2k}$ to $\mathbb C^{2k}$ preserving the standard complex symplectic form, and $\mathrm{U}_{2k}$ is the group of unitary $2k\times 2k$-matrices. An $\mathrm{USp}_{2k}$-bundle is a complex vector bundle of rank $2k$ together with a special family of trivialisations such that the transition functions between any two trivialisations in this family lie in $\mathrm{USp}_{2k}$. Note that an $\mathrm{USp}_{2k}$-bundle admits a natural complex symplectic form, namely the one given by the standard complex symplectic form in any such trivialisation. We will say that a vector bundle automorphism of an $\mathrm{USp}_{2k}$-bundle is symplectic if it preserves this symplectic form or, equivalently, if it is represented by a complex symplectic matrix in each trivialisation.  \

If in addition, the transition matrices are of the special form 
$$ 
\begin{bmatrix} 
A &  0\\
0 & B 
\end{bmatrix}
$$
then there is a natural splitting $V=L_1\oplus L_2$ of $V$ by two Lagrangian subspaces, and in 
this case we will refer to $V$ as a Lagrangian $\mathrm{USp}_{2k}$-bundle. (Note that in this case we necessarily 
have that $B=(A^T)^{-1}$.)

\begin{theorem}\label{thm:3}
Let $X$ be a locally finite finite dimensional CW-complex  and let $\pi:V\rightarrow X$ be an $\mathrm{USp}_{2k}$-bundle over $X$. Let $S$
be a nullhomotopic symplectic vector bundle automorphism of $V$. Then there exist unipotent symplectic vector bundle automorphisms
$E_1, \ldots, E_N$ such that 
$$
S=E_N\circ\cdot\cdot\cdot\circ E_1.
$$
Moreover, for all $j$ we have that 
$(E_j-\mathrm{id})^2=0$ and $\mathrm{rank}(E_j-\mathrm{id})\in\{0,1,2\}$ and if $V$ is Lagrangian, we may achieve that the $E_j$'s respect the Lagrangian splitting (see Remark \ref{splitting} below).
\end{theorem}

\begin{remark}
In fact, we can achieve that each $E_j$ in Theorem~\ref{thm:3} is of the form

\begin{equation}
\label{eq:elementary_map}
E_j(u) = u + \omega(u,v)\cdot w + \omega(u,w)\cdot v
\end{equation}
for some $v,w\in V$ such that $\omega(v,w)=0$.
\end{remark}
\begin{remark}\label{splitting}
By saying that each $E_j$ respects the Lagrangian splitting, we mean that when expressing 
each $E_j$ as in \eqref{eq:elementary_map} we have that 

either $v,w\in L_1$ or $v,w\in L_2$.
\end{remark}

\begin{remark}
The fact that $V$ is an $\mathrm{USp}_{2k}$-bundle gives us both a symplectic form and a compatible Hermitian metric. This is crucial in our proof of Theorem~\ref{thm:3}. It would be interesting to see to what extent Theorem~\ref{thm:3} can be extended to the more general setting of $\mathrm{Sp}_{2k}(\mathbb C)$-bundles, where we have a symplectic form but possibly no compatible Hermitian metric. 
\end{remark}

We now turn to the case of real symplectic vector bundles. We will denote the group of real symplectic $2k\times 2k$-matrices by $\mathrm{Sp}_{2k}(\mathbb R)$.

\begin{theorem}\label{thm:RealSymp}
Let $X$ be a finite CW-complex and let $\pi:V\rightarrow X$ be an $\mathrm{Sp}_{2k}(\mathbb R)$-bundle over $X$. Let $S$
be a nullhomotopic symplectic vector bundle automorphism of $V$. Then there exist unipotent symplectic vector bundle automorphisms
$E_1, \ldots, E_N$ such that 
$$
S=E_N\circ\cdot\cdot\cdot\circ E_1.
$$
Moreover, for all $j$ we have that 
$(E_j-\mathrm{id})^2=0$ and $\mathrm{rank}(E_j-\mathrm{id})\in\{0,1,2\}$ and if $V$ is Lagrangian, we may achieve that the $E_j$'s respect the Lagrangian splitting.
\end{theorem}
\begin{remark}
Note that $\mathrm{Sp}_2=\mathrm{SL}_2$, hence Theorem~\ref{thm:RealSymp} covers the case of a special vector bundle automorphism of a rank 2 real vector bundle over a compact manifold.  
\end{remark}
\begin{remark}
We make a brief remark about the compactness assumption in Theorem \ref{thm:RealSymp}. The reason our methods fail in the non-compact case is that the first fundamental group of $\mathrm{U}_k=\mathrm{Sp}_{2k}(\mathbb R)\cap \mathrm{O}_{2k}$ is infinite, preventing us from applying Theorem~\ref{thm:CalderSiegel}. It might be tempting to look for a generalisation of Theorem~\ref{thm:RealSymp} to the non-compact setting. However, for $k=1$ the theorem is known to be false in the non-compact setting (see \cite{Vas88}). Moreover, (generalising the construction in \cite{Vas88}), we expect that the following null-homotopic maps from $\mathbb R$ to $\mathrm{Sp}_{2k}(\mathbb R)$ for any positive $k$ provide examples which can't be written as a product of unipotent factors:
$$ t \mapsto 
\begin{bmatrix} 
A(t) &  0\\
0 & A(t) 
\end{bmatrix}
$$
where $A(t)$ is the orthogonal matrix 
$$
\begin{bmatrix} 
\cos(t) &  -\sin(t) & 0 & \ldots & 0 \\
\sin(t) & \cos(t) & 0 & \ldots & 0 \\
0 & 0 & 1 & \ldots & 0 \\
\vdots & \vdots & \vdots &  \ddots & \vdots \\
 0 & 0 & 0 & \ldots & 1 
\end{bmatrix}.
$$
\end{remark}

\subsection{Compact Kähler manifolds}
Let $X$ be a compact Kähler manifold. Then, since the holonomy of a Kähler manifold is in the symplectic group, we get that $\exp(R(U,V))$ is a symplectic vector bundle automorphism of $TX$ (see Section~\ref{sec:Proofs} for details). Applying Theorem~\ref{thm:RealSymp} to this gives the following corollary:
\begin{corollary}\label{cor:Kahler}
Let $X$ be a compact Kähler manifold, $R(\cdot,\cdot)$ be its Riemann tensor, $\nabla$ its Levi-Civita connection and $U,V$ be two smooth real vector fields on $X$. Then the symplectic vector bundle automorphism of $TX$ given by 
$$ \exp(R(U,V)) = \exp(\nabla_U \nabla_V - \nabla_V \nabla_U - \nabla_{[U,V]})$$
admits a factorisation in terms of unipotent symplectic vector bundle automorphisms. 
\end{corollary}

\subsection{Gromov's Vaserstein problem}
Given a holomorphic map $A$ from a Stein manifold $X$ to $\mathrm{SL}_k(\mathbb C)$ it is natural to ask if $A$ can be written as a product of maps into the set of elementary matrices, or (closely related) if $A$ can be written as a product of maps into unipotent subgroups of $\mathrm{SL}_k(\mathbb C)$. This question, sometimes referred to as Gromov's Vaserstein problem, was suggested by Gromov in \cite{Gro89} as a possible application of his h-principle introduced in the same paper. In 2012 the question was settled by Ivarsson and Kutzschebauch \cite{IK12}. They proved that any holomorphic null-homotopic $\mathrm{SL}_k(\mathbb C)$-valued map on a finite dimensional reduced Stein space can be factorized into a product of holomorphic maps into unipotent subgroups of $\mathrm{SL}_k(\mathbb C)$. The proof was based on Vaserstein's result in \cite{Vas88} together with a refinement of Gromov's h-principle due to Forstneri\v{c} \cite{For10}. 

We would like to propose the following generalization of Gromov's Vaserstein problem:
\begin{problem}
Let $X$ be a Stein manifold and $V$ a holomorphic vector bundle over $X$. Assume $S$ is a null-homotopic holomorphic vector bundle automorphism of $V$ of determinant 1. Can $S$ be factored into a composition of unipotent holomorphic vector bundle automorphisms of $V$?
\end{problem}
A possible application of the present paper is to use Theorem~\ref{thm:1} above together with a similar application of the h-principle as in \cite{IK12} to provide a solution to this problem. 

\subsection{Factorization in terms of exponentials}
The following related factorization problem was treated in \cite{DK19, KS19, MR18}: Let $R$ be a commutative unital ring (for example the ring of holomorphic functions on the disc or, more generally, the space of holomorphic functions on a Stein space). Given an element $S$ in the identity-component of $\mathrm{SL}_k(R)$ or $\mathrm{GL}_k(R)$, are there elements $E_1,\ldots,E_n$ in their respective Lie algebras $\mathfrak{sl}_k(R)$ and $\mathfrak{gl}_k(R)$ such that 
$$ S = \exp(E_N)\circ \ldots \circ \exp(E_1)? $$
In fact, the setup in the present paper suggests a generalization of this problem to automorphisms of vector bundles and Theorem~\ref{thm:1} -- Theorem~\ref{thm:RealSymp} has consequences for this problem. See Remark~\ref{rem:ProdOfExp} in Section~\ref{sec:Proofs} for details. 

\subsection{Overview of proofs}
In most of the paper we will consider a vector bundle automorphism $S$ as a section of the endomorphism bundle $V\otimes V^*$. As a first step we will introduce an auxiliary metric on $V$; an inner product in the case when $V$ is a real vector bundle and a Hermitian form in the case when $V$ is a complex vector bundle. This defines a fiber bundle over $X$ consisting of the points in $V\otimes V^*$ that represent orthogonal maps in the case of a real vector bundle and unitary maps in the case of a complex vector bundle. We will then apply a Gram-Schmidt like process to reduce the factorization problem to the case when $S$ takes values in this fiber bundle (Proposition~\ref{splitstep1}). In particular, this means that $S$ can be considered as a section of a fiber bundle with compact fibers. The next ingredient in the proof is a general theorem on homotopies of sections of fiber bundles with compact fibres proved in Section~\ref{sec:CalderSiegel} (Theorem~\ref{thm:CalderSiegel}). While we assume that $S$ is homotopic to the identity, this theorem allows us to conclude that there is a \emph{uniform} homotopy between $S$ and the identity. Using a simple argument we can then reduce the problem to the case when $S$ is close to the identity (see the first paragraph of the proof of Theorem~\ref{thm:1} and Theorem~\ref{thm:2} in Section~\ref{sec:Proofs}). Theorem~\ref{thm:1} then follows by applying a Gauss-Jordan type process (Proposition~\ref{splitstep2}). 

Theorem~\ref{thm:2} will follow by observing that all of the above can be done while choosing local frames that respect the splitting $\mathcal L$. Theorem~\ref{thm:3} and Theorem~\ref{thm:RealSymp} will follow by a similar argument as in the case of Theorem~\ref{thm:1}, using the symplectic Gram-Schmidt and Gauss-Jordan processes defined in \cite{IKL19}. However, the compactness assumption in Theorem~\ref{thm:RealSymp} ensures existence of a uniform homotopy between $S$ and the identity without using Theorem~\ref{thm:CalderSiegel}.

The main obstacle when using the Gram-Schmidt process and the Gauss-Jordan process is that they both have to be applied locally and they are only well-defined after fixing a local frame of $V$. To achieve global constructions, we apply the Gram-Schmidt and Gauss-Jordan in an inductive manner, while building $X$ attaching 
$m$-cells to the $(m-1)$-skeleton. 
Theorem~\ref{thm:CalderSiegel} generalises a theorem by Calder and Siegel \cite{CalderSiegel78,CalderSiegel80}. The main technical ingredient in the proof is a type of deformations of maps into CW complexes constructed in Proposition~\ref{prop:CWHomotopy}.

\subsection{Organisation of the paper}
The paper is organised as follows:  Section~\ref{sec:CalderSiegel} is devoted to the proof of Theorem~\ref{thm:CalderSiegel}. The main technical ingredient in the proof, Proposition~\ref{prop:CWHomotopy} is formulated and proved in Section~\ref{sec:CW-Complexes}. Section~\ref{sec:Proofs} is devoted to the proof of Theorem~\ref{thm:1} - Theorem~\ref{thm:RealSymp}. Corollary~\ref{cor:Riemannian} and Corollary~\ref{cor:Kahler} are proved at the end of Section~\ref{sec:Proofs}.

\subsection{Acknowledgements}
The authors would like to thank Erik L\o w: First of all for explaining his work in \cite{IKL19} and then for many fruitful discussions beyond that. The first named author also want to thank Frank Kutzschebauch for an interesting and useful discussion during Nordan2019. Both authors where supported by the RCN, grant number 240569. While working on this project the first named author was also supported by the Olle Engkvist Foundation and the Knut and Alice Wallenberg Foundation. 

\section{Homotopies and uniform homotopies}\label{sec:CalderSiegel}
In this section we will consider general fiber bundles with compact structure groups over a topological space $X$. We will define a notion of \emph{uniform} homotopies of sections of such bundles. Moreover, assuming that $X$ is a finite dimensional CW-complex and that the fibers have finite first fundamental group and the homotopy type of a CW-complex of finite type, we will prove that any two homotopic sections can be joined by a uniform homotopy (see Theorem~\ref{thm:CalderSiegel}). This generalizes a theorem by Calder and Siegel \cite{CalderSiegel78,CalderSiegel80}.

\subsection{Uniform homotopies of sections}
Let $\pi: W\rightarrow X$ be a fiber bundle over a topological space $X$ with fiber $Y$ and compact structure group $G$. A homotopy between two sections $s_0$ and $s_1$ of $W$ is a continuous map $H:X\times I\rightarrow W$ such that $H(\cdot,0)=s_0$, $H(\cdot,1)=s_1$ and $H(\cdot,t)$ is a section for each $t\in [0,1]$, i.e., $\pi(H(x,t))=x$ for all $(x,t)\in X\times I$. We will let $C(I:W)$ denote the space of maps $f:I\rightarrow W$ such that $\pi\circ f$ is constant, i.e., such that there is $x_0\in X$ with $\pi(f(t))=x_0$ for all $t\in I$. Then a homotopy $H$ of sections of $W$ induces a map $h:X\rightarrow C(I:W)$, given by $h(x)= H(x,\cdot)$. In the rest of this section we will use this map, rather than $H$, to refer to a homotopy. 

If $(U,\rho)$ is a local trivialization of $W$, then composition with $\rho$ induces an invertible map $\tilde \rho: C(I:W|_U)\rightarrow C(I:(U\times Y))$. Moreover, letting $\pi_2:U\times Y\rightarrow Y$ be the projection on the second factor we get that composition with $\pi_2$ induces a map 
$\tau: C(I:(U\times Y))\rightarrow C(I,Y)$. Given a homotopy $h:X\rightarrow C(I:W)$ and letting $h_U=\tau \circ \tilde\rho\circ h|_U$ we get the following commutative diagram
\vskip 1cm

\begin{center}
\begin{tikzpicture}
  \matrix (m) [matrix of math nodes,row sep=3em,column sep=4em,minimum width=2em]
  {
    C(I:W|_U) & C(I:(U\times Y))  &  C(I,Y) \\
   U & &   \\};
    \path[-stealth]
(m-1-1) edge node [above] {$\tilde\rho$} (m-1-2)
(m-1-2) edge node [above] {$\tau$} (m-1-3)
(m-2-1) edge node [left] {$h|_U$} (m-1-1)
(m-2-1) edge node [below] {$h_U$} (m-1-3);
 \end{tikzpicture}
 \end{center}

\vskip 1cm

The map $h_U$ can be thought of as a local representative of $h$ in the trivialization $(U,\rho)$. Note that the structure group $G$ of $W$ acts on $C(I,Y)$ by composition. Given $x\in U$ and $t\in I$ we have that $h(x)(t)\in W$ and $h_U(x)(t)\in Y$ and $h|_U$ can be recovered from $h_U$ by
$$ h(x)(t) = (\tau\circ\tilde\rho)^{-1}(x,h_U(x)(t)).$$
Moreover, let $(U',\rho')$ be another local trivialization of $W$ and assume $x\in U\cap U'$. Then $\pi_2\circ \rho' \circ h|_U(x) = g\circ \pi_2\circ \rho \circ h|_U(x)$ for some element $g$ in the structure group $G$. It follows that 
$$ h_{U'}(x) = g \circ h_{U}(x). $$
In particular, fixing a $G$ invariant set $K\subset C(I,Y)$, we get that $h_U(x)\in K$ if and only if $h_{U'}(x)\in K$. 
\begin{definition}
Let $\pi: W\rightarrow X$ be a fiber bundle over a topological space $X$ with fiber $Y$ and compact structure group $G$. Then a homotopy of sections of $W$ is a \emph{uniform} homotopy if there is a $G$-invariant compact set $K\subset C(I,Y)$ such that for any $x\in X$ and some (hence any) local trivialization $(\rho,U)$ of $W$ over a neighbourhood $U$ of $x$ we have $h_U(x)\in K$. 
\end{definition}

\begin{lemma}\label{lem:ArcelaAscoli}
Let $h$ be a uniform homotopy of a fiber bundle $W$ as above. Assume in addition that the fiber $Y$ is a metric space with distance function $d:Y\times Y \rightarrow \mathbb{R}_{\geq 0}$. Then the following holds: For any $\epsilon>0$, there is a finite set of real numbers $0=t_1<t_2<\ldots<t_N=1$ such that for any $x\in X$ and any trivialization over a neighbourhood $U$ of $x$, we have $d(h_U(x)(t),h_U(x)(t'))<\epsilon$ whenever $t,t'\in [t_i,t_{i+1}]$ for some $i$. 
\end{lemma}
\begin{proof}
By by definition, there is a compact $G$-invariant set $K\subset C(I,Y)$ such that for any $x\in X$ and any local trivialization $(U,\rho)$ of $W$ over a neighbourhood $U$ of $x$, $h_U(x)\in K$. By the Arzela-Ascoli Theorem $K\subset C(I,Y)$ is equicontinuous. This proves the lemma. 
\end{proof}

\subsection{Existence of uniform homotopies}
From now on we will assume that $X$ is a CW-complex of dimension $d<\infty$ and that $Y$ has the homotopy type of a CW-complex of finite type. The main theorem of this section is the following theorem:
\begin{theorem}\label{thm:CalderSiegel}
Let $X$ be a locally finite CW-complex of finite dimension, and $W\rightarrow X$ be a fiber bundle over $X$ with compact structure group and fiber $Y$. Assume also that $Y$ has the the homotopy type of a CW-complex of finite type and 
that $\pi_1(Y)$ is finite.
Furthermore, let $h:X\times I\rightarrow W$ be a homotopy of sections. Then there is a uniform homotopy of sections
$\tilde h:X\times I\rightarrow W$ such that $\tilde h(\cdot,0)=h(\cdot,0)$ and $\tilde h(\cdot,1)=h(\cdot,1)$.
\end{theorem}

Before we prove Theorem~\ref{thm:CalderSiegel}, we note that by Lemma~\ref{lem:ArcelaAscoli} we have the following corollary:
\begin{corollary}
Assume in addition that the fiber $Y$ of $W$ is a metric space with distance function $d:Y\times Y \rightarrow \mathbb{R}_{\geq 0}$. Then there is a homotopy $\tilde h$ between $h_0$ and $h_1$, with the property that for any $\epsilon>0$ there is a finite set of real numbers $0=t_1<t_2<\ldots<t_N=1$ such that for any $x\in X$ and any trivialization over a neighbourhood $U$ of $x$, we have $d(h_U(x)(t),h_U(x)(t'))<\epsilon$ whenever $t,t'\in [t_i,t_{i+1}]$ for some $i$.
\end{corollary}

The main technical ingredient in the proof of Theorem~\ref{thm:CalderSiegel} is a result on homotopies of maps into CW-complexes proved in Section~\ref{sec:CW-Complexes}. We will summarize the consequences of this needed to prove Theorem~\ref{thm:CalderSiegel} in Lemma~\ref{lemma:DeformToSkeleton} below. To state Lemma~\ref{lemma:DeformToSkeleton} we will first introduce some notation.

Let $p:C(I,Y)\rightarrow Y\times Y$ be the map given by $p(f)=(f(0),f(1))$. Given a homotopy $h:X\rightarrow C(I:W)$ and a local trivialization $(U,\rho)$ of $W$, we have the following commutative diagram:
\vskip 1cm
\begin{center}
\begin{tikzpicture}
  \matrix (m) [matrix of math nodes,row sep=3em,column sep=4em,minimum width=2em]
  {
      &  C(I,Y) \\
    U & Y\times Y  \\};
    \path[-stealth]
(m-2-1) edge node [above] {$h_U$} (m-1-2)
(m-1-2) edge node [right] {$p$} (m-2-2)
(m-2-1) edge node [below] {$p\circ h_U$} (m-2-2);
 \end{tikzpicture}
 \end{center}
\vskip 1cm
Note that if $h$ is a homotopy between two sections $s_0, s_1$ of $W$, then $h':X\rightarrow C(I,W)$ is also a homotopy between $s_0$ and $s_1$ if and only if $$p\circ h_U = p\circ h'_U$$
for any local trivialization $(U,\rho)$ of $W$. 

We will now fix a point in $Y$ and let $\Omega Y$ denote the pointed loop space of $Y$. By standard theory, there is an open covering $\{B_i\}_{i=1}^\infty$ of $Y\times Y$ such that for each $i$, the fibration $p|_{p^{-1}(\overline B_i)}:p^{-1}(\overline B_i) \rightarrow \overline B_i$ is homotopy equivalent to the trivial fibration $\pi_1:\overline B_i \times \Omega Y\rightarrow \overline B_i$.

\begin{lemma}\label{lemma:DeformToSkeleton}
Let $U\subset X$ and $p:C(I,Y)\rightarrow Y\times Y$ be the fibration above. Assume that $\overline U$ is contained in a finite subcomplex 
$X'\subset X$ and that 
$h_U:X'\rightarrow C(I,Y)$ is a continuous map such that $B=p\circ h_U(X')\subset B_i$ with $B_i$ as above. 
Let $A\subset U$ and assume $h_U(A)$ is contained in a compact set $K\subset C(I,Y)$. Then there is a compact set $K'\subset C(I,Y)$ depending only on $d$, $K$ and $i$, and a fiber homotopy $H:U\times I\rightarrow C(I,Y)$ such that $H(\cdot,0) = h_U$,
$$ p\circ H(\cdot,t) = p \circ h_U  $$
for all $t\in [0,1]$ and 
$$ H(U\times \{1\})\cup H(A\times I) \subset K'. $$
\end{lemma}
\begin{remark}
Lemma~\ref{lemma:DeformToSkeleton} above is closely related to Lemma~2 in \cite{CalderSiegel80}, stating that a CW-complex of finite type satisfies the \emph{relative compressibility property}. However, Lemma~\ref{lemma:DeformToSkeleton} above is stronger in that the set $K'$ only depends on $d$, $K$ and $B$. In other words, given $d$, $K$ and $B$ we have that $K'$ is independent of $h_U$ and $U$. This will be crucial when applying it to prove Theorem~\ref{thm:CalderSiegel}.
\end{remark}

\begin{proof}
By assumption, the fibration $p: p^{-1}(\overline B)\rightarrow \overline B$ is homotopy equivalent to the trivial fibration $\pi_1:\overline B \times \Omega Y\rightarrow \overline B$. 

Let $\pi_2: \overline B \times \Omega Y\rightarrow\Omega Y$ be the projection onto the second factor. Moreover, let 
$$ \phi: p^{-1}(\overline B)\rightarrow \overline B\times \Omega Y, \;\;\; \psi:\overline B\times \Omega Y\rightarrow p^{-1}(\overline B)$$ 
be a homotopy equivalence and $F:p^{-1}(\overline B)\times I\rightarrow p^{-1}(\overline B)$ be a homotopy between the identity map on $p^{-1}(\overline B)$ and $\psi\circ\phi$. By \cite{Wall65}, $\Omega Y$ is homotopic to a CW-complex of finite type. Fix such a homotopy equivalence and for positive integers $m$, let $(\Omega Y)^m$ be the $m$-skeleton of $\Omega Y$ induced by this homotopy equivalence. By compactness of $\pi_2\circ \phi(K)\subset \Omega Y$, we get that $\pi_2\circ\phi(K)$ is contained in $(\Omega Y)^{m_K}$ for some $m_K$. Applying Proposition~\ref{prop:CWHomotopy} in Section~\ref{sec:CW-Complexes} we get a homotopy $\Theta:U\times I\rightarrow \Omega Y$ such that $\Theta(\cdot,0) = \pi_2\circ \phi\circ h_U$ and the following inclusions hold:
\begin{eqnarray}
\Theta(U\times \{1\}) & \subset & (\Omega Y)^d \nonumber \\
\Theta(A\times I) & \subset & (\Omega Y)^{m_K}. \nonumber
\end{eqnarray}

We will let $\gamma$ denote the map $p\circ h_U$. Let $H$ be given by
$$ 
H(x,t) = 
\begin{cases}
F(h_U(x),2t), & t \in [0,1/2] \\
\psi\left(\gamma(x), \Theta(x,2t-1)\right), & t \in (1/2,1].
\end{cases}
$$
This means $H(\cdot,0) = F(\cdot,0)\circ h_U = h_U$. Continuity of $H$ follows from
$$ F(h(x),1) = \psi\circ\phi\circ h_U(x) = \psi \left(\gamma(x),\Theta(x,0)\right). $$
Moreover, 
\begin{eqnarray}
H(A,I) & \subset & F(h_U(A)\times I) \cup \psi \left(\gamma(A)\times\Theta(A,I)\right) \nonumber \\
& \subset & F(K\times I) \cup \psi \left(\overline B\times (\Omega Y)^{m_K}\right) \nonumber
\end{eqnarray}   
and
\begin{eqnarray}
H(U\times \{1\}) & \subset & \psi\left(\gamma(U)\times \Theta(U,1)\right) \nonumber \\
& \subset & \psi\left(\overline B\times(\Omega Y)^d\right) \nonumber
\end{eqnarray}
which are both compact since the CW structure of $\Omega Y$ is of finite type. Note that $F$ and $\psi$ are determined by $B$. Letting 
$$ K' = F(K\times I) \cup \psi \left(\overline B\times (\Omega Y)^{\max\{d,m_K\}}\right)  $$
proves the lemma. 
\end{proof}

We are now ready to prove Theorem~\ref{thm:CalderSiegel}.
\begin{proof}[Proof of Theorem~\ref{thm:CalderSiegel}]
Let $\{B_i\}_{i=1}^m$ be a finite open covering of $Y\times Y$ such that for each $i$ we have that $p^{-1}(\overline {B_i})$ is homotopy equivalent to $\overline {B_i}\times \Omega Y$.
We may assume that  $\{U_j\}_{j=1}^\infty$ is an open covering of $X$ such that the following holds
\begin{itemize}
    \item[(1)] For each $j$ we have that $\overline{U_j}$ is compact, and there 
     is a trivialization of $W$ over $U_j$ such that $p\circ h_{U_j}(U_j)\subset B_{i_j}$ 
    for some $i_j\in \{1,\ldots,m\}$. 
    \item[(2)] There is a family of $d+1$ index sets $I_1,\ldots,I_{d+1}\subset \mathbb{N}$, such that if $i,j\in I_k$ for some $k$, then $U_i$ and $U_j$ are disjoint, and $\cup_{k=1}^{d+1}I_k=\mathbb N$. 
\end{itemize}

The first point above is obtained by first picking a family of trivialisations $\{(U_j,\rho_j)\}_{j\in\mathbb N}$ of $W$ such that 
$\overline{U_j}$ is compact for all $j$. 
 We may then further 
consider the refinement of $\{U_j\}_{j\in\mathbb N}$ given by $\{U_j\cap \left( p\circ h_{U_j}\right)^{-1}(B_i)\}_{i,j}$. By possibly passing to a refinement of this second covering, the second point can be obtained by applying Ostrand's theorem on coloured dimension (see Lemma~3 in \cite{Ost71}). 

Finally, let $\{V_j\}_{j=1}^\infty$ be an open covering of $X$ such that $\overline{V_j}\subset U_j$ for each $j$, and for each $k\in \{1,\ldots,d+1\}$ define 
$$ U^k = \bigcup_{i\in I_1\cup\ldots\cup I_k} U_i, \qquad V^k = \bigcup_{i\in I_1\cup\ldots\cup I_k} V_i. $$

Set $h^0:=h$ and $K_0=V^0=U^0=\emptyset$. We will proceed by induction. Assume that $h^{k-1}$ is a homotopy matching 
$h^0$ at $\{0,1\}$, and that 
$K_{k-1}\subset C(I,Y)$ is a $G$-invariant compact set such that for any $x\in V^{k-1}$ and any trivialization of $W$ in a neighbourhood $U$ of $x$, we have that

$h^{k-1}_U(x)\in K_{k-1}$. 
We will now find a homotopy $h^k$ matching $h^{k-1}$ at $\{0,1\}$ and a $G$-invariant compact set $K_k\subset C(I,Y)$ such that for any $x\in V^{k}$ and any trivialization of $W$ in a neighbourhood $U$ of $x$, we have that 
$h^{k}_U(x)\in K_{k}$.

By our choice of covering $\{U_j\}_j$, we have that $W$ is trivial over each $U_j$ and $p\circ h_{U_j}(U_j)\subset B_{i_j}$ for some $i_j\in \{1,\ldots,m\}$. In other words, $p:C(I,Y)\rightarrow Y\times Y$ is trivial over $p\circ h_{U_j}(U_j)$.

For each $j\in I_k$  we let 
$$ 
H_j:U_j\times I \rightarrow C(I:Y)
$$ 
be the fiber homotopy given by applying Lemma~\ref{lemma:DeformToSkeleton} to the map 
$$
h^{k-1}_{U_j}:U_j \rightarrow C(I:Y)
$$ 
and the set $A_k= U_j\cap V^{k-1}$. 

We get that 
$$
H_j(A_k\times I)\cup H_j(U_j\times\{1\})
$$ 
is contained in a compact set $K^{i_j}_k\subset C(I,Y)$ which is determined by $d$, $K_{k-1}$ and $B_{i_j}$ (note that $i_j\in \{1,\ldots,m\}$).

Let $\eta_k:X\rightarrow I$ be a continuous map such that 
$$ 
\eta_k^{-1}(1)=\overline{ \bigcup_{j\in I_k} V_j} \qquad \eta_k^{-1}(0)=X\setminus \bigcup_{j\in I_k} U_j. 
$$ 
For any $x\in X\setminus \bigcup_{j\in I_k} U_j$ we put $h^k(x):=h^{k-1}(x)$. 
As $U_l$ and $U_j$ are disjoint for any distinct $l,j\in I_k$ we get for any $x\in \bigcup_{j\in I_k} U_j$, a unique $j\in I_k$ such that $x\in U_{j}$.
 
We define $h^k(x)$ by

$$ 
h^k_{U_j}(x) = H(x,\eta_k(x))
$$
It follows that for any $x\in U_j\cap V^{k-1}$, we have that $h^k_{U_j}(x)$ is contained in 
$ H(U_j\cap V^{k-1}\times I)$ and for any $x\in V_j$, we have that $h^k_{U_j}(x)$ is contained in $H(V_j\times \{1\})$ both of which are contained in $K_k^{i_j}$. Performing this procedure for all $j\in I_k$ we get that for any $x\in V^k$ and any trivialization $(U,\rho)$ of $W$ in a neighbourhood $U$ of $x$, we have that $h_U^k(x)$ is in the compact set
$$ K_k := \bigcup_{i=1}^m GK_k^{i}, $$
where $GK_k^{i}$ denotes the $G$-orbit of $K_k^{i}\subset C(I,Y)$ which is compact by compactness of $G$. The theorem then follows by induction on $k\in \{1,\ldots,d+1\}$. In particular, $h^{d+1}$ is the desired uniform homotopy. 

\end{proof}

\section{Deformation of maps into CW Complexes}
\label{sec:CW-Complexes}

This goal of this section is to state and prove Proposition~\ref{prop:CWHomotopy} below, which is used in the previous section to prove Lemma~\ref{lemma:DeformToSkeleton}. We begin by recalling the definition of a CW complex.
\begin{itemize}
\item[(1)] Start with a discrete set $Y^0$, the zero cells of $Y$.
\item[(2)] Inductively, form the $n$-skeleton $Y^n$ from $Y^{n-1}$ by 
attaching $n$-cells $e_\alpha^n$ via maps 
$\varphi_\alpha:S^{n-1}\rightarrow Y^{n-1}$. This means that $Y^n$
is the quotient space of $X^{n-1}\sqcup_{\alpha} D_\alpha^n$ under 
the identifications $x\sim \varphi_\alpha(x)$ for $x\in\partial D^n_\alpha$.
The cell $e^n_\alpha$ is the homeomorphic image of $D^n_\alpha\setminus\partial D^n_\alpha$
under the quotient map. 
\item[(3)] $Y=\cup_n Y^n$ is equipped with the weak topology: A set $A\subset Y$ is open (or closed) if and only if 
$A\cap Y^n$ is open (or closed) in $Y^n$ for each $n$. 
\end{itemize}

The goal of this section is to prove the following proposition. The proposition is used in the proof of Lemma~\ref{lemma:DeformToSkeleton}. 

\begin{proposition}\label{prop:CWHomotopy}
Let $X$ and $Y$ be CW complexes, with $\mathrm{dim}(X)=d$, let $L\subset X$ be compact, and 
let $f:\Omega\rightarrow Y$ be a continuous map where $\Omega$ is an open set containing $L$.
Then there exists an open set $K\subset\Omega '\subset\Omega$ and 
a homotopy 
$\Theta: [0,1]\times\Omega'\rightarrow Y$ such that the following holds. 
\begin{itemize}
\item[(1)] $\Theta(x,0)=f(x)$ for all $x\in \Omega'$, 
\item[(2)] If $f(x)\in \overline{e^m_\alpha}$ then $\Theta(x,t)\in e^m_\alpha\cup Y^{m-1}\subset Y^m$ for all $t\in [0,1]$, 
\item[(3)] $\Theta(x,1)\in Y^d$ for all $x\in \Omega'$.
\end{itemize}
\end{proposition}
\begin{remark}
Proposition~\ref{prop:CWHomotopy} is closely related to Theorem~1 on page 199 in \cite{Sch68} stating that any map between CW-complexes is homotopic to a cellular mapping. However, we have not found a specific reference when it comes to compact subsets of CW-complexes, and not an explicit statement analogous to (2). 
\end{remark}

Let $X$ and $Y$ be CW-complexes. Note that for a cell $e^n_{X,\alpha}\subset X$, a cell $e^m_{Y,\beta}\subset Y$, 
and a closed set $S\subset e^n_{X,\alpha}$ there is a well defined notion of smooth maps from $S$ into 
$e^m_{Y,\beta}$ via the homeomorphisms induced by the quotient maps. (By convention, 
if $S\subset\mathbb R^n$ and if $f:S\rightarrow\mathbb R^m$ is a map, we say that $f$ is smooth 
if for any point $x\in S$ there is an open neighbourhood $U_x$ of $X$ in $\mathbb R^n$ and 
a smooth map $\tilde f:U_x\rightarrow\mathbb R^m$ with $\tilde f|_{S\cap U_x})=f$.)

\medskip

Since each $e^n_\alpha$ naturally identified with $D^n_\alpha$ we may use this identification to 
consider convex sets in $e^n_\alpha$, and more generally balls $rB^n_\alpha\subset e^n_\alpha$ of radius 
$r$ centred at the origin.

\begin{definition}
Let $X$ and $Y$ be CW complexes, and let $S\subset X$. 
We say that a continuous map $f:X\rightarrow Y$ is smooth 
at $x\in S$, if $x\in e^n_{X,\alpha}, f(x)\in e^m_{Y,\beta}$, and 
if $f$ is smooth on $U_x\cap S$ for some open neighbourhood of $U_x$ of $x$.   
We say that $f$ is smooth on $S$ if $f$ is smooth at each $x\in S$. 
\end{definition}

We now give a proof of Proposition \ref{prop:CWHomotopy} granted Proposition \ref{prop:smooth}
which we will prove below. 

\emph{Proof of Proposition \ref{prop:CWHomotopy}:}
Since $L$ is compact we have that $f(L)$ intersects finitely many cells in $Y$ and 
in particular the set of such cells $e^m_{Y}$ with $m>d$ is finite. 
Order these cells in a way $e^{m_j}_Y$, $j=1,...,N$ such that $m_{j+1}\leq m_j$
for all $j$. We proceed by induction on $j$. 

Choose a compact convex set $K_1\subset e^{m_1}_Y$ with interior, and 
an open set $U_1$ containing $K_1$. 
Applying Proposition \ref{prop:smooth} we may assume that $f$ is smooth on $f^{-1}(K_1)$
and so by Sard's Theorem, we may assume that 
there is a point in $e^{m_1}_Y$ which is not in $f(X)$, and identifying $e^{m_1}_Y$
with $D^{m_1}$ we may assume that this point is the origin. In $D^{m_1}$
we define 
$$
\psi_t(x):=\frac{x}{\|x\|}(\|x\| + t(1-\|x\|)),
$$
and we set $f_t=\psi_t\circ f$ on $f^{-1}(e^m_Y)$, and $f_t=f$ elsewhere.  
Repeating this construction, for $j=2,3,..$, and 
concatenating gives the proposition. 
$\hfill\square$

\begin{proposition}\label{prop:smooth}
Let $X$ and $Y$ be CW complexes, let $L\subset X$ be compact,  
and let $f:\Omega\rightarrow Y$ 
be a continuous map where $\Omega$ is an open set containing $L$. 
Let $e^m_{Y,\alpha}$ be a cell in $Y$, 
let $K\subset e^m_{Y,\alpha}$ be a compact convex set, and let $U\subset e^m_{Y,\alpha}$
be an open set containing $K$. 
Then there exists a an open set $L\subset\Omega'\subset\Omega$ and a 
homotopy $F:[0,1]\times \Omega'\rightarrow Y$ with $F_0=f$ and $F_1$ 
smooth on $F_1^{-1}(K)$. 
Moreover, we have that $F_t=f$ on $\Omega'\setminus f^{-1}(U)$, $F_t(f^{-1}(U))\subset U$,  
and  if $f(x)\in Y^n$ for any $x\in X$, then $F_t(x)\in Y^n$ for all $t\in [0,1]$. 
\end{proposition}
\begin{proof}
Without loss of generality we may assume that $U$ is convex. 
We proceed by induction on $n$, and our induction hypothesis $I_n$ is that we have constructed
an isotopy $F^n_t$ with the desired properties near, in the induced topology on $X^n$, 
$L_n=L\cap X^n$, where $X^n$ denotes the $n$-skeleton of $X$, and moreover, setting 
$$
E_n=\overline{\{x\in X^n:F^n_t(x)\neq f(x) \mbox{ for some } t\}}, 
$$
we have that $\overline{F^n_t(E_n)}\subset U$.  Then $I_0$ holds automatically setting $F^0_t=f$ for all $t$. 
Now,  $L_{n+1}$ is constructed from $L_n$ by attaching intersections $L\cap e^{n+1}_{X,\beta}$ for $(n+1)$-cells $e^{n+1}_{X,\beta}$ with $\beta\in J_{n+1}$.
For a fixed $\beta$, we let $\tilde f_t$ denote the restriction of $F^n_t$ to an open neighbourhood of $L\cap\partial e^{n+1}_{X,\beta}$. Then 
Lemma \ref{homext} below applies to extend the homotopy $\tilde f_t$ to $\overline{e^{n+1}_{X,\beta}}$, in a 
way such that Lemma \ref{homsmooth} applies to homotopize further to get a homotopy $F^{(n+1),\beta}_t$ satisfying 
the desired differentiability conditions, as well as $\overline{F^{(n+1),\beta}_t}(E_{n+1}^\beta)\subset U$. 

\end{proof}

\begin{lemma}\label{homext}
Let $e^n_{X,\beta}$ be a cell in $X$, let $L\subset\overline{e^n_{X,\beta}}$ be a closed set, and let $\Omega\subset\overline{e^n_{X,\beta}}$
be an open set containing $L$. 
Let $e^m_{Y,\alpha}$ be a cell in $Y$, let $U\subset\subset e^m_{Y,\alpha}$ be an open convex set, and let 
$$
f:\Omega\rightarrow Y
$$
be a continuous map. 
Let $\tilde f_t$ be a homotopy of continuous maps 
$$
\tilde f_t:[0,1]\times (\Omega\cap\partial e^n_{X,\beta})\rightarrow Y
$$
with $\tilde f_0= f$, set 
$$
E=\overline{\{x\in (\Omega\cap\partial e^n_{X,\beta}):f_t(x)\neq f(x) \mbox{ for some } t\}}, 
$$
and assume that $\overline{\tilde f(E)}\subset U$. 
Then 
there exists an open set $L\subset\Omega'\subset\Omega$ and 
a homotopy of continuous maps 
$$
F:[0,1]\times \Omega'\rightarrow Y
$$ 
with $F|_{\Omega'\cap\partial e^n_{X,\beta}}=\tilde f$,  $F_t=f$ on $\Omega'\setminus f^{-1}(U)$, and 
$F_t(f^{-1}(U))\subset U$. 
\end{lemma}
\begin{proof}
We have attaching maps $\varphi_\alpha:\partial D_\alpha^m\rightarrow Y$ and $\varphi_\beta:\partial D_\beta^n\rightarrow X$, 
and by construction of a CW-complex, both are regarded to extend as homeomorphisms on the interior 
of the respective balls.  Abusing notation we let $L$ denote $\varphi_\alpha^{-1}(L)$ and $\Omega$ denote $\varphi_\alpha^{-1}(\Omega)$, 
and interpret $f$ as a continuous map 
$$
f:\overline D_\beta^n\rightarrow Y, 
$$
and $\tilde f_t$ as an isotopy $\tilde f_t:\partial D^n\rightarrow Y$, with $E\subset\partial D_\beta^n$. 
Moreover, near $E$, $\tilde f_t$ may be thought of as a map into the convex open set $U\subset D_\alpha^m$. \

Let $W\subset\partial D_\beta^n$ be an open set with $E\subset W$ such that 
$\tilde f_t(W)\subset\subset U$ for all $t$, and choose $\delta>0$ 
such that 
$f(W_\delta)\subset\subset U$, where 
$$
W_\delta=\{x\in\overline D_\beta^n: \|x\|>1-\delta, \frac{x}{\|x\|}\in W\}
$$
Let $\chi_\delta(x)=\chi_\delta(\|x\|)$ be in $C(\overline{D_\beta^n})$, with $0\leq \chi_\delta\leq 1$, 
$\mathrm{Supp}(\chi_\delta)\subset \{1-\delta<\|x\|\leq 1\}$, and $\chi_\delta(x)=1$ for $\|x\|=1$. Define 
$$
F_t(x) = \chi_\delta(x)\cdot\tilde f_t(\frac{x}{\|x\|}) + (1-\chi_\delta(x))\cdot f(x)
$$
for $x\in W_\delta$, and $F_t=f$ on an open neighbourhood of $L\setminus W_\delta$.
\end{proof}

\begin{lemma}\label{homsmooth}
Let $e^n_{X,\beta}$ be a cell in $X$, let $L\subset\overline{e^n_{X,\beta}}$ be a closed set, and let $\Omega\subset\overline{e^n_{X,\beta}}$
be an open set containing $L$. 
Let $e^m_{Y,\alpha}$ be a cell in $Y$, let $K\subset e^m_{Y,\alpha}$ be a compact convex set, and 
let 
$$
f:\Omega\rightarrow Y
$$
be a continuous map.  Let $U\subset\subset e^m_{Y,\alpha}$
be a convex open set with $K\subset U$.
Then there exists an open set $L\subset\Omega'\subset\Omega$ and 
a homotopy 
$$
F:[0,1]\times \Omega'\rightarrow Y
$$
of continuous maps such that 
\begin{itemize}
\item[(i)] $F_0=f$, 
\item[(ii)] $F_t=f$ on $\Omega'\cap\partial e^n_{X,\beta}$ for all $t\in [0,1]$, and 
\item[(iii)] $F_1$ is smooth on $F_1^{-1}(K)$. 
\end{itemize}
Moreover, setting 
$$
E=\overline{\{x\in\Omega':F_t(x)\neq f(x) \mbox{ for some } t\}}, 
$$
we have that 
\begin{itemize}
\item[(iv)] $\overline{F(E)}\subset U$. 
\end{itemize}
\end{lemma}
\begin{proof} 
Set $\tilde U=f^{-1}(U)$, let $W\subset\subset U$ be a convex open set with $K\subset W$, 
and set $\tilde W=f^{-1}(W)$.  
Let $\chi$ be a smooth function on $\overline {D_\beta^n}$ with 
$\mathrm{Supp}(\chi)\subset\tilde W, 0\leq\chi\leq 1$ and $\chi\equiv 1$ near $f^{-1}(K)$. 
Let $f_j$ be a sequence of continuous functions on $\overline{\tilde W}$ such that each $f_j$ 
is smooth on $\tilde W\cap B^n$, such that $f_j\rightarrow f$ uniformly on $\tilde W$, 
and such that $f_j=f$ on $\tilde W\cap\partial B^n$.  Set 
$$
F_j = f + \chi\cdot (f_j - f)
$$
on $\tilde W$, and $F_j=f$ on $\overline D^n\setminus\tilde W$. 
Finally we set $F_{j,t}(x)= (1-t)\cdot f + t\cdot F_j$ and observe that if we set $F_t=F_{j,t}$ for large 
enough $j$, then the conclusion of the lemma is satisfied. 
\end{proof}

\section{Proofs of the main theorems}\label{sec:Proofs}

\subsection{Vector bundle automorphisms as sections of fiber bundles}
Let $W^{\mathrm{GL}}$ be the fiber bundle over $X$ consisting of the points in the endomorphism bundle $V\otimes V^*$ that represent invertible maps on the fibers of $V$. 
%Alternatively, $W^{\mathrm{GL}}$ can be constructed in the following manner: 

%There is an open cover $\{U_j\}_{j\in\mathbb N}$ of $X$
%and trivialisations $\psi_j:V_j:=\pi^{-1}(U_j)\rightarrow U_j\times\mathbb R^k$ (resp. $\mathbb C^k$), and the transition 
%maps $\psi_{ij}$ from $V_j$ to $V_i$ over $U_{ij}:=U_i\cap U_j$ are determined by a cocycle of maps $\psi_{ij}:U_{ij}\rightarrow \mathrm{GL}_k(\mathbb R)$ or $\psi_{ij}:U_{ij}\rightarrow \mathrm{GL}_k(\mathbb C)$.

%Then $W^{\mathrm{GL}}$ is the fiber bundle whose fiber $W^{\mathrm{GL}}_x$
%over a point $x\in X$ is the group of isomorphisms of $V_x$. 
%The transition map from $(W^{\mathrm{GL}}_x)^i$ to
%$(W^{\mathrm{GL}}_x)^j$ over $U_{ij}$ is given by 
%$$
%B_{ji}(x)(A_i(x))=  \psi_{ji}(x)\circ A_{i}(x)\circ\psi_{ij}(x).
%$$

We also get the following fiber bundles over $X$: 
\begin{itemize}
\item[(1)] $W^{\mathrm{SL}}$ denoting the fiber bundle over $X$ consisting of the elements $W^{\mathrm{GL}}$ that has determinant one. 
\item[(2)] $W^{\mathrm{u}}$ denoting the fiber bundle over $X$ consisting of the (fiberwise) \emph{unipotent} elements $W^{\mathrm{GL}}$. 
\end{itemize} 
If $U\subset X$ is an open set and $W$ any of the fiber bundles above we denote the space of continuous 
sections of $W|_U$ by $\Gamma(U,W)$. Note that a vector bundle automorphism $S$ of $V$ defines a section in $W^{\mathrm{GL}}$. In the rest of the paper we will identify $S$ with this section and hence consider $S$ an element in $\Gamma(X,W^{\mathrm{GL}})$. 
\begin{remark}
Note that a vector bundle automorphism $S$ is unipotent (i.e. $(S-\mathrm{Id})^m=0$ for some positive integer $m$) if and only if the associated section of $W^{\mathrm{GL}}$ lies in $W^{\mathrm{u}}$. This follows from the elementary fact that any unipotent automorphism $s$ of a vector space of dimension $k$ satisfies $(s-\mathrm{Id})^k=0$. 
\end{remark}
\begin{remark}
\label{rem:ProdOfExp}
The fiber bundles $W^{\mathrm{GL}}$ and $W^{\mathrm{SL}}$ are bundles of Lie groups. The fiberwise Lie algebras of these bundles form bundles of Lie Algebras $W^{\mathfrak{gl}}$ and $W^{\mathfrak{sl}}$ with exponential maps
\begin{equation}
   W^{\mathfrak{gl}} \xrightarrow{\mathrm{exp}} W^{\mathrm{GL}} \textnormal{ and } W^{\mathfrak{sl}} \xrightarrow{\mathrm{exp}} W^{\mathrm{SL}}
\end{equation}
(see for example \cite{SilvaWeinstein}, Section 16.3). Composing with the exponential map we get that any section $E$ of $W^{\mathfrak{sl}}$ or $W^{\mathfrak{gl}}$ defines a section $\mathrm{exp}(E)$ of $W^{\mathrm{SL}}$ or $W^{\mathrm{GL}}$.
It is thus natural to ask the following questions generalizing questions treated in  \cite{DK19,KS19,MR18}: Given a nullhomotopic section $S$ of $W^{\mathrm{SL}}$ or $W^{\mathrm{GL}}$, is it possible to find sections $E_1,\ldots,E_N$ of $W^{\mathfrak{sl}}$ or $W^{\mathfrak{gl}}$ such that
$$ S = \mathrm{exp}(E_N)\circ \ldots \circ \mathrm{exp}(E_1).$$
In fact, any unipotent automorphism is the image under the exponental map of a section of the corresponding bundle of Lie algebras. More precisely, putting 
$$ s = -(\mathrm{Id}-S) - \frac{(\mathrm{Id}-S)^2}{2} - \ldots - \frac{(\mathrm{Id}-S)^{k-1}}{k-1} $$
gives (using $(S-\mathrm{Id})^k=0$) that $\mathrm{exp}(s)=S$ (see also Lemma~1 in~\cite{DK19}). Consequently, Theorem~\ref{thm:1} -- Theorem~\ref{thm:RealSymp} have consequences for the problem of exponential factorization.
\end{remark}

\subsection{Auxiliary metrics and local frames for Theorem~\ref{thm:1} and Theorem~\ref{thm:2}} \label{sec:AuxiliaryMetrics}

In the case when $V$ is a real vector bundle we fix an inner product on $V$. We may then 
choose the local trivialisations $\psi_j:\pi^{-1}(U_j)\rightarrow U_j\times\mathbb R^k$
such that the pullback of the inner product under $(\psi_j)$ is the standard Euclidean on $\mathbb R^k$ for all $x\in U_j$. In that case the transition functions
$\psi_{ij}$ all take values in the orthogonal group $\mathrm{O}_k$. Similarly, if $V$ is a complex vector bundle we fix a Hermitian form on $V$ and pick trivialisations such that the transition functions take values in the unitary group $\mathrm{U}_k$.

These are both compact groups, in other words $W^{\mathrm{GL}}$ becomes a fiber bundle with compact structure group.  In the proof of Theorem \ref{thm:1} we 
will always assume that $V$ is equipped with this structure. This will be crucial when applying Theorem~\ref{thm:CalderSiegel}. Moreover, the auxiliary metric above also give us two fiber bundles over $X$. In the case when $V$ is a real vector bundle: 
\begin{itemize}
    \item[(3)] $W^{\mathrm{SO}}$, denoting the fiber bundle over $X$ consisting of the elements in $W^{\mathrm{SL}}$ that are orthogonal with respect to the inner product on $V$
\end{itemize}
and in the case when $V$ is a complex vector bundle:
\begin{itemize}
    \item[(3)] $W^{\mathrm{SU}}$,  denoting the fiber bundle over $X$ consisting of the elements in $W^{\mathrm{SL}}$ that are unitary with respect to the Hermitian form on $V$
\end{itemize}
The significance of these fiber bundles is that their fibers are compact. This will also be crucial when applying Theorem~\ref{thm:CalderSiegel}.

\medskip

For $f\in\Gamma (X,W^{\mathrm{GL}})$, the inner product on $V$ (resp. the Hermitian form) induce a fiberwise operator norm 
$$
\|f\|_x:=\sup_{v\in V_x\setminus\{0\}}\frac{|f(x)(v)|}{|v|},
$$
where $|\cdot|$ is the norm induced by the inner produce (resp. Hermitian form), and a distance function $d_W:\Gamma(X,W^{\mathrm{GL}})\times\Gamma(X,W^{\mathrm{GL}})\rightarrow [0,\infty]$
by 
$$
d_W(f,g):=\sup_{x\in X}\|f-g\|_x.
$$
The space of bounded sections, i.e., $f\in\Gamma(X,W^{\mathrm{GL}})$ such that 
$$\|f\|_W :=d_W(f,0)<\infty,$$ 
is thus 
given the structure of a normed space which we denote by $\Gamma_b(X,W^{\mathrm{GL}})$. In particular, $\Gamma(X,W^{\mathrm{SO}})$ (and $\Gamma(X,W^{\mathrm{SU}})$ respectively) is naturally included in $\Gamma_b(X,W^{\mathrm{GL}})$.\

\medskip

If in addition $V$ admits a splitting $\mathcal L=\{L_1,...,L_k\}$ into line bundles, we may 
choose the 
inner product (or Hermitian form, respectively)
such that the $L_i$'s are mutually orthogonal. We may also achieve that each $\psi_j$ sends each $L_i$ to the line generated by the $i$'th standard basis vector in $\mathbb R^n$ ($\mathbb C^n$, respectively). Then the transition maps 
$\psi_{ij}(x)$ are diagonal matrices whose entries in the real case take values in $\{-1,1\}$, and 
in the complex case in $\{z\in\mathbb C:|z|=1\}$. This will be used in the proof of Theorem~\ref{thm:2}.\

\subsection{Gauss-Jordan, Gram-Schmidt and the proofs of Theorem~\ref{thm:1} and Theorem~\ref{thm:2}}

The main ingredients in the proofs of Theorem~\ref{thm:1} and Theorem~\ref{thm:2} are
Theorem~\ref{thm:CalderSiegel} on uniform homotopies together with Proposition~\ref{splitstep1} and Proposition~\ref{splitstep2}. Proposition~\ref{splitstep1} is related to the Gram-Schmidt process and says that a section in $W^{\mathrm{SL}}$, or more generally a continuous family of sections $S_t, t\in [0,1]$ of $W^{\mathrm{SL}}$ can after composing with a number of unipotent sections be assumed to take values in $W^{\mathrm{SO}}$ (or $W^{\mathrm{SU}}$, respectively). 

\begin{proposition}[The Gram-Schmidt process]\label{splitstep1}
Let $S_t, t\in [0,1]$ be a homotopy of sections in $\Gamma(X,W^{\mathrm{SL}})$. Then there exist 
homotopies of unipotent sections 
$E^t_j, j=1,...,N$ in $\Gamma(X,W^{\mathrm{u}})$ with $(\mathrm{id}-E^t_j)^2=0$ and $\mathrm{rank}(\mathrm{id}-E^t_j)\in\{0,1\}$, 
such that in the case when $V$ is a real vector bundle
$$
E^t_N\circ\cdot\cdot\cdot\circ E^t_1\circ S_t\in \Gamma(X,W^{\mathrm{SO}})
$$ 
and in the case when $V$ is a complex vector bundle
$$
E^t_N\circ\cdot\cdot\cdot\circ E^t_1\circ S_t\in \Gamma(X,W^{\mathrm{SU}})
$$ 
for all $t\in [0,1]$. Moreover, if $S_0=\mathrm{id}$ then $E^0_j=\mathrm{id}$
for $j=1,...,N$.
\end{proposition}

Proposition~\ref{splitstep2} is related to the Gauss-Jordan process and says that any $S\in \Gamma_b(X,W^{\mathrm{SL}})$ which is sufficiently close to the identity can be factorised into a product of unipotent sections. 

\begin{proposition}[The Gauss-Jordan process]\label{splitstep2}
There exists an $\epsilon>0$ such that if $S\in\Gamma_b^{\mathrm{SL}}(X,W)$
satisfies $d_W(S,\mathrm{id})<\epsilon$, then there exist 
unipotent sections 
$E_j\in\Gamma^{\mathrm{u}}(X,W)$ for $j=1,...,N$, such that
$$
E_N\circ\cdot\cdot\cdot\circ E_1\circ S=\mathrm{id}.
$$ 

\end{proposition}

Given Proposition~\ref{splitstep1} and Proposition~\ref{splitstep2}, Theorem~\ref{thm:1} and Theorem~\ref{thm:2} are proved in the following manner:
\medskip

\emph{Proof of Theorems  \ref{thm:1} and Theorem \ref{thm:2}:}
Let $S_t, t\in [0,1]$, be a homotopy between $\mathrm{id}$ and $S$. By Proposition \ref{splitstep1}
we may assume that $S_t\in\Gamma(X,W^{\mathrm{SO}})$ for all $t\in [0,1]$ if $V$ is a real vector bundle and $S_t\in\Gamma(X,W^{\mathrm{SU}})$ for all $t\in [0,1]$ if $V$ is a complex vector bundle. By the discussion in Section~\ref{sec:AuxiliaryMetrics}, $W^{\mathrm{SL}}$ has compact structure group. To apply Theorem~\ref{thm:CalderSiegel} it thus suffices to verify that $W^{\mathrm{SO}}$ and $W^{\mathrm{SU}}$ have compact fibers with finite first fundamental group. But this is immediate since the fibers are isomorphic to $SO_k$ and $SU_k$ respectively, and $\pi_1(SO_k)$ is finite for $k\geq 3$ and $\pi_1(SU_k)$ is finite for any $k$. 
After applying Theorem~\ref{thm:CalderSiegel} (or rather its corollary) we may assume that the homotopy $S_t$ is uniform. 
Hence, there is a partition $0=t_0<t_1<\cdot\cdot\cdot <t_{m}=1$, such that $\tilde S_j:=S_{t_{j+1}}\circ S_{t_j}^{-1}$
satisfies $d_W(\tilde S_j,\mathrm{id})<\epsilon$ for $j=0,...,m-1$, where $\epsilon$ is the constant 
from Proposition \ref{splitstep2}. Since
$$ S = \tilde S_{m-1}\circ \tilde S_{m-2}\circ \ldots \circ \tilde S_0$$
the proof of Theorem \ref{thm:1} is completed by an 
application of Proposition \ref{splitstep2} to each $\tilde S_j$, and then composing. Finally, the 
proof of Theorem~\ref{thm:2} is immediate, since in that case, all the local coordinates used 
in the proofs of Lemma~\ref{extcell} and Lemma~\ref{extskel} below respect the splitting $L=\{L_1,...,L_k\}$
(see the last paragraph in Section~\ref{sec:AuxiliaryMetrics}).
$\hfill\square$

\medskip

\subsection{Proofs of Proposition~\ref{splitstep1} and Proposition~\ref{splitstep2}}

We start by giving two lemmas. 

\begin{lemma}\label{extcell}
Let $E^t:[0,1]\times \partial D^n\rightarrow SL_k(\mathbb C)$ (resp. $SL_k(\mathbb R)$) be continuous, with $(\mathrm{Id}-E^t)^2=0$
and $\mathrm{rank}(\mathrm{Id}-E^t)\in\{0,1\}$ for all $t$.  Then $E^t$ extends to 
a continuous map $\tilde E^t:[0,1]\times \overline{D^n}\rightarrow SL(k)$, with 
$(\mathrm{Id}-\tilde E^t)^2=0$
and $\mathrm{rank}(\mathrm{Id}-\tilde E^t)\in\{0,1\}$ for all $t$. Moreover we may ensure that 
$$
\|(\mathrm{Id}-\tilde E^t_x)\| \leq \sup_{y\in\partial D^n} \{\|(\mathrm{Id}-\tilde E^t_y)\|\}
$$
for all $x\in D^n$. 
\end{lemma}
\begin{proof}
Let $S_t\subset\partial D^n$ be the set such that $E^t=\mathrm{id}$. Then for each 
$x\in\partial D^n\setminus S_t$ we have that 
$$
E^t_x(v)= v+ \lambda(t,x)(v)\cdot w(t,x), 
$$
with $\lambda(t,x)\in (\mathbb R^k)^*$, where $\lambda$ is normalized such that $\|\lambda(t,x)\|=1$ for all $(t,x)$.  
Let $\pi:D^n\setminus\{0\}$ denote the projection $\pi(x)=\frac{x}{\|x\|}$.  
First, for all $(t,x)$ with $\pi(x)\in S_t$ we set $\tilde E^t_x=\mathrm{id}$.  Next, let $\chi\in C^\infty([0,1])$
with $\chi=0$ near $[0,1/2]$ and $\chi=1$ near $\{1\}$, and for each $(t,x)$ with $\pi(x)\notin S_t$
set 
$$
\tilde E^t_x (v)=  v+ \chi(\|x\|)\cdot \lambda(t,\pi(x))(v)\cdot w(t,\pi(x)), 
$$
\end{proof}

\begin{lemma}\label{extskel}
Let $X$ be a CW-complex and let $W\rightarrow X$ be a $GL_k(\mathbb C)$-bundle (resp. $GL_k(\mathbb R)$). 
Let $E^t_m\in\Gamma (X^m,W|_{X^m})$ be continuous with $(\mathrm{Id}-E^t_m)^2=0$
and $\mathrm{rank}(\mathrm{Id}-E^t_m)\in\{0,1\}$ for all $t\in [0,1]$, where $X^m$ denotes the $m$-skeleton of $X$.  
Then $E^t_m$ admits a continuous  extension $E^t\in\Gamma (X,W)$ with $(\mathrm{Id}-E^t)^2=0$
and $\mathrm{rank}(\mathrm{Id}-E^t)\in\{0,1\}$ for all $t\in [0,1]$. Moreover, if $E^t_m\in\Gamma_b(X^m,W|_{X^m})$
with $d_W(E^t_m,\mathrm{Id})<\epsilon$ for all $t$, then $d_W(E^t,\mathrm{Id})<\epsilon$ for all $t$. 
\end{lemma}
\begin{proof}
For each cell $e^{m+1}_\alpha$ attached to $X^m$ we have that $E^t_m$ has a suitable extension to $e^{m +1}_\alpha$
by the previous lemma, and the distance to the identity is controlled.  Hence, we have a suitable extension to $X^{m+1}$, and 
we may proceed by finite induction on the skeleton. 
\end{proof}

\emph{Proof of Proposition \ref{splitstep1}:}
We start by, for each $0$-cell, applying the Gram-Schmidt process with the parameter $t$ to the restriction of $S_t$ to the $0$-cell. 
This gives sections $E^t_{0,1},...,E^t_{0,N}$ satisfying the conditions of Lemma \ref{extskel}, and so that 
$$
E^t_{0,1}\circ...\circ E^t_{0,N}\circ (S_t|_{X_0})\in\Gamma(X_0,W|_{X_0}^{SO/SU}). 
$$
By Lemma \ref{extskel} we have that each $E^t_{0,j}$ extends to an element in $\Gamma(X,W^{SL})$.  \

We proceed by induction, and assume that there are $E^t_{j}\in\Gamma(X,W^{SL})$ such that 
\begin{equation}\label{compk}
E^t_{1}\circ...\circ E^t_{N}\circ (S_t|_{X_k})\in\Gamma(X_k,W|_{X_k}^{SO/SU}). 
\end{equation}
For each $(k+1)$-cell $\overline{e^{k+1}_\alpha}$ we apply the Gram-Schmidt process to the composition \eqref{compk}
in the local coordinates given by $\overline{D^{k+1}_\alpha}$. This gives sections $E^t_{\alpha,j}$ such that 
$$
E^t_{\alpha,1}\circ...\circ E^t_{\alpha,N}\circ E^t_{0,1}\circ...\circ E^t_{0,N}\circ (S_t|_{\overline{e^{k+1}_\alpha}})\in
\Gamma(X_{\overline{e^{k+1}_\alpha}},W|_{\overline{e^{k+1}_\alpha}}^{SO/SU}).
$$
By \eqref{compk} we have that each $E^t_{\alpha,j}$ is the identity map on $\partial e^{k+1}_\alpha$ and so 
extends to the identity map on $W|_{X_k}$. So the collection $\{E^t_{\alpha,j}\}$ does the job on $W|_{X_k}$, 
and by Lemma \ref{extskel} they globalise to suitable elements of  $\Gamma(X,W^{SL})$. $\hfill\square$

\medskip

\emph{Proof of Proposition \ref{splitstep2}:}
The proof is analogous the proof of Proposition \ref{splitstep1}, but applying the Gauss-Jordan process instead of Gram-Schmidt. 
The first key point to keep in mind, is that in order to apply Gauss-Jordan to a map $E:\overline{D^k}\rightarrow SL$ we need to assume 
that $\|E_x-\mathrm{Id}\|$ is sufficiently small in order to give an algorithm that works simultanuously for all $x$. 
This is ensured 
at each step of the inductive process by the initially chosen sufficiently small $\epsilon$, and control in norms 
when we extend using Lemma \ref{extskel}. The second point to keep in mind, is that passing from $X_k$
attaching a $(k+1)$-cell $e^{k+1}_\alpha$, all maps in play are already the identity map on $\partial e^{k+1}_\alpha$, 
and so the Gauss-Jordan process produced only identity maps there, that can all be extended to $X_k$.
$\hfill\square$

\subsection{Proofs of Theorem \ref{thm:3} and Theorem \ref{thm:RealSymp}}

The proof of Theorem 3 is analogous to the that in the previous section, but we need to make some remarks.  \

Firstly, consider the Hermitian metric $g$ on the $USp_{2k}$-bundle $V$ induced by the standard Hermitian inner product
on $\mathbb C^{2k}$. In the inductive construction, when attaching $m$-cells $\varphi:\overline{D^m}\rightarrow X$, we 
consider the $USp_{2k}$ bundle $\tilde V=\varphi^*V\rightarrow\overline{D^m}$.  Then there is a finite open cover $U_1,...,U_N$
of $\overline{D^m}$ with transition maps $A_{ij}:U_{ij}\rightarrow USp_{2k}$. In order to trivialize $\tilde V$ in a way that respects 
the metric $g$ we then need to trivialize $\tilde V$ finding maps $A_j:U_j\rightarrow USp_{2k}$ such that 
\begin{equation}\label{trans}
A_i = A_j \circ A_{ij}
\end{equation}
on $U_{ij}$.  But this amounts to finding a section of the associated $USp_{2k}$-bundle whose transition maps are 
defined precisely by \eqref{trans}, and such a section always exists on the contractible space $\overline{D^m}$. \

Secondly, note that when preforming the Gram-Smith or Gauss-Jordan process for symplectic matrices 
using elementary symplectic matrices, they are not on the form $u\mapsto u+\lambda(u)\cdot v$
with $\lambda(v)=0$ as in Lemma \ref{extcell}. In this case, as was shown in \cite{IKL19}, the symplectic elementary transformations 
are of the form
$$
u\mapsto u + \omega(u,v)\cdot w + \omega(u,w)\cdot v, 
$$
where $w\in\mathrm{Ker}(\omega(v,\cdot))$ and $v\in\mathrm{Ker}(\omega(w,\cdot))$.
Adapting the proof using these elementary maps is straight forward. \

Finally, note that the proof of Theorem \ref{thm:RealSymp} is simpler because of the compactness assumption on $X$, and 
there is no need for an initial application of Gram-Smith to transform into the orthogonal group, in order to be able to use 
the result on uniform homotopies.

\subsection{Proofs of Corollary~\ref{cor:Riemannian} and Corollary~\ref{cor:Kahler}} 
\begin{proof}[Proof of Corollary~\ref{cor:Riemannian}]
By Theorem~\ref{thm:1} it suffices to verify that $\exp(R(U,V))$ is a nullhomotopic special vector bundle automorphism. First of all, noting that
$$\exp(R(0,0))=\exp(0)=Id$$
we see that a homotopy from the identity map to $\exp(R(U,V))$ is given by
$$ t\mapsto S_t = \exp(R(tU,tV)). $$
Moreover, to see that $\exp(R(U,V))$ has determinant one we need to show that $R(U,V)$ lies in the Lie algebra of $\mathrm{SL}_k(\mathbb R)$, in other words that the trace of $R(U,V)$ is zero. A conceptual proof of this is given by the fact that $R(U,V)$ arise as the variation of the holonomy of small quadrilaterals on $X$ (the restricted holonomy of a Riemannian manifold lies in $SO_k$).  For a more direct proof, let $\langle \cdot,\cdot\rangle$ be the inner product on $TX$ given by the Riemannian metric. One of the basic symmetries of the Riemann tensor is that
\begin{equation} \langle R(U,V) Y, Z \rangle = - \langle Y, R(U,V) Z \rangle. \label{eq:SkewSymmetric} \end{equation}
Writing this in a frame which is orthogonal with respect to the Riemannian metric we get that the matrix representing $R(U,V)$ is skew-symmetric. In particular its trace is zero. 
\end{proof}

\begin{proof}[Proof of Corollary~\ref{cor:Kahler}]
By the same argument as in the proof of Corollary~\ref{cor:Riemannian} it suffices to verify that $\exp(R(U,V))$ is a symplectic vector bundle automorphism, in other words that $R(U,V)$ lies in the Lie algebra of the symplectic group. As in the proof of Corollary~\ref{cor:Riemannian}, since the holonomy of a Kähler manifold lies in $\mathrm{U}_k = \mathrm{Sp}_{2k}(\mathbb R)\cap \mathrm{O}_{2k}$, a conceptual proof is given by the fact that $R(U,V)$ arise as the variation of the holonomy of small quadrilaterals on $X$. For a more direct proof, note that the Lie algebra of $\mathrm{Sp}_{2k}(\mathbb R)$ is given by the group of matrices $A$ satisfying 
\begin{equation}
    \Omega A + A^T\Omega = 0 \label{eq:SympLieAlg}
\end{equation} 
where $\Omega$ is the $2k\times 2k$-matrix given by
$$ 
\Omega = 
\begin{bmatrix} 
0 &  Id_k\\
-Id_k & 0 
\end{bmatrix}.
$$
Now, since the complex structure $J$ on $X$ is parallel with respect to the Levi-Civita connection of the Riemannian metric, we get that 
\begin{eqnarray} R(U,V)J & = & \left(\nabla_U\nabla_V - \nabla_V\nabla_U - \nabla_{[U,V]}\right) J \nonumber \\
& = & J \left(\nabla_U\nabla_V - \nabla_V\nabla_U - \nabla_{[U,V]}\right) \nonumber \\
& = & JR(U,V). \end{eqnarray}
Writing this identity in a frame which is orthogonal with respect to the Riemannian metric and in which $J$ is represented by the standard matrix $\Omega$ gives $ \Omega A = A \Omega $
where $A$ is the matrix representing $R(U,V)$. Since $A$ is also skew-symmetric by \eqref{eq:SkewSymmetric}, we get 
$$\Omega A = A\Omega = -A^T \Omega$$
and \eqref{eq:SympLieAlg} follows. 
\end{proof}

\end{document}